\documentclass[a4paper,11pt]{article}
\usepackage[T1]{fontenc}
\usepackage{lmodern,amsmath,amsthm,amsfonts,amssymb,graphicx,float,microtype,thmtools,underscore,mathtools,xurl,thm-restate}
\usepackage[dvipsnames,svgnames,table]{xcolor}
\usepackage[shortlabels]{enumitem}
\setlist[itemize]{topsep=0ex,itemsep=0ex,parsep=0ex}
\setlist[enumerate]{topsep=0ex,itemsep=0ex,parsep=0ex}
\usepackage{pstricks,auto-pst-pdf}
\usepackage[unicode=true]{hyperref}
\hypersetup{
colorlinks,
breaklinks=true,
linkcolor={blue!60!black},
citecolor={black},
urlcolor={blue!60!black},
pdftitle={Tree-Decompositions with Small Width, Spread, Order and Degree}}
\usepackage[capitalise, compress, nameinlink, noabbrev]{cleveref}
\crefname{lem}{Lemma}{Lemmas}
\crefname{thm}{Theorem}{Theorems}
\crefname{cor}{Corollary}{Corollaries}
\newcommand{\defn}[1]{\textcolor{Maroon}{\emph{#1}}}
\newcommand{\mathdefn}[1]{\textcolor{Maroon}{#1}}
\usepackage[longnamesfirst,numbers,sort&compress]{natbib}
\makeatletter
\def\NAT@spacechar{~}
\makeatother
\usepackage[tmargin=30mm,bmargin=30mm,lmargin=30mm,rmargin=30mm]{geometry}
\renewcommand{\baselinestretch}{1.1}
\allowdisplaybreaks

\DeclarePairedDelimiter{\ceil}{\lceil}{\rceil}

\renewcommand{\epsilon}{\varepsilon}
\renewcommand{\emptyset}{\varnothing}

\renewcommand{\geq}{\geqslant}
\renewcommand{\leq}{\leqslant}

\DeclareMathOperator{\tw}{tw}

\newcommand{\RR}{\mathbb{R}}

\newcommand{\TT}{\mathcal{T}}

\renewcommand{\thefootnote}{\fnsymbol{footnote}}
\theoremstyle{plain}
\newtheorem{thm}{Theorem}
\newtheorem{lem}[thm]{Lemma}
\newtheorem{cor}[thm]{Corollary}

\crefname{obs}{Observation}{Observations}
\newtheorem*{lem*}{Lemma}
\theoremstyle{definition}

\newtheorem*{conj*}{Conjecture}

\begin{document}
\title{\bf\boldmath\fontsize{18pt}{20pt}\selectfont 
Tree-Decompositions with \\
Small Width, Spread, Order and Degree}

\author{David~R.~Wood\,\footnotemark[2]}

\maketitle

\begin{abstract}
Tree-decompositions of graphs are of fundamental importance in structural and algorithmic graph theory. The main property of tree-decompositions is the width (the maximum size of a bag minus 1). We show that every graph has a tree-decomposition with near-optimal width, where each vertex appears in few bags. In particular, every graph with treewidth $k$ has a tree-decomposition with width at most $14k+13$, where each vertex $v$ appears in at most $\deg(v)+1$ bags. This improves an exponential bound by Ding and Oporowski [1995] to linear, and establishes a conjecture of theirs in a strong sense. In a second result, we show that every graph with treewidth $k$ has a tree-decomposition with width at most $3k-1$, where on average each vertex appears in at most three bags.
\end{abstract}

\footnotetext[2]{School of Mathematics, Monash University, Melbourne, Australia (\texttt{david.wood@monash.edu}). Research supported by the Australian Research Council and by NSERC. }

\renewcommand{\thefootnote}{\arabic{footnote}}
\section{Introduction}

Tree-decompositions were introduced by \citet{RS-II}, as a key ingredient in their Graph Minor Theory. Indeed, the dichotomy between minor-closed classes with or without bounded treewidth is a central theme of their work. Tree-decompositions arise in several other results, such as the Erd\H{o}s-P\'osa theorem for planar minors~\cite{RS-V,CvBHJR19}, and  Reed's beautiful theorem on $k$-near bipartite graphs~\citep{Reed99a}. Tree-decompositions are also a key tool in algorithmic graph theory, since many NP-complete problems are solvable in linear time on graphs with bounded treewidth~\citep{Courcelle90}. 

For a non-empty tree $T$, a \defn{$T$-decomposition} of a graph\footnote{We consider simple undirected graphs $G$ with vertex set $V(G)$ and edge set $E(G)$. Let \defn{$\Delta(G)$} be the maximum degree of $G$. Let \defn{$\omega(G)$} be the number of vertices in a largest clique in  $G$. A graph $G$ is \defn{empty} if $V(G)=\emptyset$. } $G$ is a collection $(B_x:x \in V(T))$ such that:
\begin{itemize}
    \item $B_x\subseteq V(G)$ for each $x\in V(T)$ (each $B_x$ is called a \defn{bag}),
    \item for each edge ${vw \in E(G)}$, there is a node ${x \in V(T)}$ with ${v,w \in B_x}$, and
    \item for each vertex ${v \in V(G)}$, the set $\{ x \in V(T) : v \in B_x \}$ induces a non-empty (connected) subtree of $T$.
\end{itemize}
The \defn{width} of such a $T$-decomposition is ${\max\{ |B_x| : x \in V(T) \}-1}$. A \defn{tree-decomposition} is a $T$-decomposition for any tree $T$. The \defn{treewidth} of a graph $G$, denoted \defn{$\tw(G)$}, is the minimum width of a tree-decomposition of $G$. Treewidth\footnote{Equivalent notions to treewidth were introduced by \citet{BB72} and \citet{Halin76} prior to the work of Robertson and Seymour.} is the standard measure of how similar a graph is to a tree. Indeed, a connected graph has treewidth at most 1 if and only if it is a tree. See \citep{HW17,Bodlaender98,Reed97} for surveys on treewidth.

The main property of tree-decompositions is the width. However, much recent work has looked at other properties~\citep{Adler06}, including chromatic number of the bags~\citep{Seymour16,HK17,BFMMSTT19,HRWY21}, independence number of the bags~\citep{DMS21,DMS24a,DMS24b,DFGKM24,DKKMMSW24,MR22,AACHSV24}, diameter of the bags \citep{CDN16,Lokshtanov10,BS24,DG07}, and treewidth of the bags~\citep{LNW}. This paper studies three other properties of tree-decompositions. 

\subsubsection*{Spread}

\citet{DO95} introduced the following definition (motivated by connections to the congestion and dilation of graph embeddings). The \defn{spread} of a vertex $v$ in a tree-decomposition $(B_x:x\in V(T))$ is the number of nodes $x\in V(T)$ such that $v\in B_x$. If a vertex $v$ has spread $s$ in a tree-decomposition with width $k$, then $\deg(v)\leq sk$. So if $s$ is a constant, then the width must increase with the maximum degree. Conversely, \citet{BodEng-JAlg97} and \citet{DO95} independently showed that every graph with treewidth $k$ and maximum degree $\Delta$ has a tree-decomposition with width at most some function $f(k,\Delta)$, where every vertex has spread at most 2 (called a \defn{domino tree-decomposition}). The best known bound here is $f(k,\Delta)=(9k+7)\Delta(\Delta+1)-1$, due to  \citet{Bodlaender-DMTCS99}. 

To avoid dependence on maximum degree, our focus is on tree-decompositions where the spread of a vertex $v$ is allowed to depend on $\deg(v)$. 
However, \citet{DO95} showed that one cannot hope for tree-decompositions with optimal width and spread bounded even by a function of the maximum degree. In particular, they constructed, for any integer $n$, a graph $G$ with treewidth 3 and maximum degree 4, such that in every tree-decomposition of $G$ with width 3 there is a vertex with spread $n$. This says that for spread to be bounded by a function of the degree, one must allow for tree-decompositions with near-optimal width. \citet{DO95} showed this is possible (albeit with large dependence on the treewidth). 

\begin{thm}[\citep{DO95}] 
\label{DO}
Every graph $G$ with treewidth $k$ has a tree-decomposition with width at most $2^{k+1}(k+1)-1$, such that each vertex $v\in V(G)$ has spread at most $2\cdot3^{2^k}\deg_G(v)+1$.
\end{thm}

\citet{DO95} conjectured that the bound on the spread in \cref{DO} can be improved to only depend on $\deg_G(v)$. We establish this conjecture, with much better dependence on $k$ in the bound on the width.

\begin{thm}
\label{Spread}
Every graph $G$ with treewidth $k$ has a tree-decomposition with width at most $14k+13$, such that each vertex $v\in V(G)$ has spread at most $\deg_G(v)+1$.
\end{thm}

We now illustrate this result with an  example. Let $G$ be the $n\times n$ grid graph. Let $(v_1,\dots,v_{n^2})$ be the ordering of $V(G)$ consisting of the first row, followed by the second row, followed by the third row, etc. 
Let $B_i:= \{v_i,\dots,v_{n+i}\}$ for $i\in\{1,\dots,n^2-n\}$. It is easily seen that $(B_1,\dots,B_{n^2-n})$ is a path-decomposition of $G$ with width $n$. So $\tw(G)\leq n$. In fact, $\tw(G)=n$ for $n\geq 2$ (proved via treewidth--bramble duality~\citep{ST93}). On the other hand, if $C_i$ is the union of the $i$-th row and the $(i+1)$-th row of $G$, then it is easily seen that $(C_1,\dots,C_{n-1})$ is a path-decomposition of $G$ with width $2n-1$. The first path-decomposition has optimal width and maximum spread $n+1$, whereas the second path-decomposition has near-optimal width and maximum spread 2. \cref{Spread} says that analogous behaviour holds for every graph.

Spread is naturally interpreted in terms of minors. It is well known that if $(B_x:x\in V(T))$ is a tree-decomposition of a graph $G$ with width $k$, then $G$ is a minor of the strong product $T\boxtimes K_{k+1}$ (see \citep{HW17}). Moreover, the number of vertices in the branch set representing $v\in V(G)$ equals the spread of $v$ in $(B_x:x\in V(T))$. \cref{Spread} thus implies:

\begin{cor}
Any graph $G$ with treewidth $k$ is a minor of $T\boxtimes K_{14k+14}$ for some tree $T$, where the branch set representing a vertex $v\in V(G)$ has at most $\deg_G(v)+1$ vertices.
\end{cor}

\cref{Spread} also has an attractive consequence for chordal completions. A graph $G$ is \defn{chordal} if $G$ has no induced cycle of length at least 4. A \defn{chordal completion} (also called \defn{triangulation}) of a graph $G$ is a chordal graph $G'$ such that $G$ is a spanning subgraph of $G'$. There is a large literature on chordal completions; see the survey by \citet{Heggernes06}.
A graph is chordal if and only if it has a tree-decomposition in which each bag is a clique (see \citep[Proposition~12.3.6]{Diestel5}). It follows from the Helly property that $\tw(Q)=\omega(Q)-1$ for every chordal graph $Q$. Say $(B_x:x\in V(T))$ is the tree-decomposition of a graph $G$ from \cref{Spread}, and $G'$ is obtained from $G$ by adding edges so that $B_x$ is a clique for each $x\in V(T)$. So $(B_x:x\in V(T))$ is a tree-decomposition of $G'$, and $G'$ is chordal. The next result follows.

\begin{cor}
\label{Completion}
Any graph $G$ with treewidth $k$ has a chordal completion $G'$ such that $\omega(G')=\tw(G')+1\leq 14k+14$, and each vertex $v\in V(G')=V(G)$ satisfies $\deg_{G'}(v)\leq \linebreak (14k+13)(1+\deg_G(v))$. 
\end{cor}

\subsubsection*{Order}

The second property of tree-decompositions that we consider is the number of bags. Define the \defn{order} of a tree-decomposition $(B_x:x\in V(T))$  to be $|V(T)|$. It is folklore that every graph $G$ with treewidth $k$ has a tree-decomposition with width $k$ and order $|V(G)|-k$ (see \citep{DvoWoo} for a proof). Every tree-decomposition of a graph $G$ with width $k$ has order at least $\frac{|V(G)|}{k+1}$. We show that this lower bound can be achieved within a small constant factor.

\begin{thm}
\label{Small}
For any graph $G$ and integer $k\geq\max\{\tw(G),1\}$, there is a tree-decomposition of $G$ with width at most $3k-1$ and order at most $\max\{\frac{|V(G)|}{k}-1,1\}$. 
\end{thm}

Note that in \cref{Small}, the total size of the bags is less than $3|V(G)|$ (assuming $|V(G)|\geq k$). That is, the average spread of a vertex is less than 3. 

\cref{Small} is reminiscent of the folklore result saying that every $k$-colourable graph $G$ is $(2k-1)$-colourable with at most $\ceil{\frac{|V(G)|}{k}}$ vertices in each colour class (see \citep{PTT99} for example). 

The proofs of \cref{Spread,Small} can be combined to give a tree-decomposition with both small spread and small order.

\begin{thm}
\label{SpreadAndSmall}
For any graph $G$ and integer $k\geq\tw(G)$,  $G$ has a tree-decomposition with width at most $56k+53$ and order at most $\max\{\frac{|V(G)|}{14k+14},1\}$, such that each vertex $v\in V(G)$ has spread at most $\deg_G(v)+1$.
\end{thm}

We emphasise that treewidth is not only of interest when it is bounded. For example, it follows from the Lipton-Tarjan separator theorem that every $n$-vertex planar graph has treewidth $O(\sqrt{n})$ (see \citep{DMW17} for a direct proof). \cref{SpreadAndSmall} implies that every such graph has a tree-decomposition with width $O(\sqrt{n})$ and order $O(\sqrt{n})$, such that each vertex $v$ has spread at most $\deg(v)+1$. More generally, \citet{AST90} showed that every $n$-vertex $K_t$-minor-free graph has treewidth at most $t^{3/2}\sqrt{n}$. \cref{SpreadAndSmall} implies that every such graph has a tree-decomposition with width $O(t^{3/2}\sqrt{n})$ and order $O(\sqrt{n}/t^{3/2})$, such that each vertex $v$ has spread at most $\deg(v)+1$. Nothing like these results are possible from \cref{DO}, because of the large dependence on $k$. 

\subsubsection*{Degree}

Define the \defn{degree} of a tree-decomposition $(B_x:x\in V(T))$ to be the maximum degree of $T$. It is well-known that every graph with treewidth $k$ has a tree-decomposition with width $k$ and degree 3. To see this, starting from a tree-decomposition with width $k$, replace each node $x\in V(T)$ by a path $P$ on $\deg_T(x)$ vertices, copy the original bag at $x$ to each node of $P$, and make each node of $P$ adjacent to exactly one of the neighbours of $x$ in $T$. This operation does not maintain small  spread. Nevertheless, the proof of \cref{Spread} is easily adapted to bound the degree with no increase in the width or spread. 

\begin{thm}
\label{SpreadDegree}
Every graph $G$ with treewidth $k$ has a tree-decomposition with width at most $14k+13$ and degree at most $6$, such that each vertex $v\in V(G)$ has spread at most $\deg_G(v)+1$.
\end{thm}




The paper is organised as follows. \cref{Separators} presents results about balanced separators that underpin the main proofs. The proofs of \cref{Spread,SpreadDegree} are presented in \cref{SmallSpread}. 
A key tool here is the notion of a `slick' tree-decomposition, which is of independent interest. \cref{Small} is proved in \cref{SmallOrder}. 

\section{Balanced Separators}
\label{Separators}

This section provides a series of results about balanced separators in graphs of given treewidth. 
We start with the following classical lemma of \citet{RS-II}.


\begin{lem}[{\protect\citep[(2.5)]{RS-II}}]
\label{RS-BasicSeparator}
For any graph $G$ with treewidth at most $k$, for any set $S\subseteq V(G)$, there is a set $X$ of at most $k+1$ vertices in $G$ such that each component of $G-X$ has at most $\frac{|S\setminus X|}{2}$ vertices in $S$. 
\end{lem}

For the proof of \cref{Spread} we need a version of \cref{RS-BasicSeparator} where each component of $G-X$ has substantially fewer than $\frac{|S|}{2}$ vertices in $S$. The next lemmas accomplish this (see \citep{DvoWoo,Thomassen88a} for similar results in the unweighted setting).

%

A \defn{weighting} of a graph $G$ is a function $\gamma:V(G)\to \RR^+$.
The \defn{weight} of a subgraph $G'$ of $G$ is $\mathdefn{\gamma(G')}:=\sum_{v\in V(G')}\gamma(v)$. 

For a tree $T$ rooted at a vertex $r\in V(T)$, any subtree $T'$ of $T$ is considered to be rooted at the (unique) vertex in $T'$ at minimum distance from $r$ in $T$. 

\begin{lem}
\label{TreeDecSep}
\label{TreewidthSep}
For any tree-decomposition $\TT$ of a graph $G$, for any weighting $\gamma$ of $G$, for any integer $q\geq 0$, there is a set $X\subseteq V(G)$ consisting of the union of at most $q$ bags of $\TT$, such that each component of $G-X$ has weight at most $\frac{\gamma(G)}{q+1}$. In particular, if $\TT$ has width $k$, then $|X|\leq q(k+1)$.
\end{lem}

\begin{proof}
We proceed by induction on $q$. The $q=0$ case holds trivially with $Z=\emptyset$. Now assume that $q\geq 1$ and the result holds for $q-1$. Say $\TT=(B_x:x\in V(T))$. Root $T$ at an arbitrary vertex $r$. For each node $v\in V(T)$, let $T_v$ be the subtree of $T$ induced by $v$ and its descendants. Let $G_v:= G[ \bigcup\{ B_x : x\in V(T_v) \}]$. If $G_r$ has weight at most $\frac{\gamma(G)}{q+1}$, then $Z=\emptyset$ satisfies the claim. Now assume 
 $G_r$ has weight greater than $\frac{\gamma(G)}{q+1}$. Let $v$ be a vertex in $T$ furthest from $r$ such that $G_v$ has weight greater than $\frac{\gamma(G)}{q+1}$. Let $T':=T-V(T_v)$ and $G':=G-V(G_v)$. So $G'$ has weight at most $\frac{q \gamma(G)}{q+1}$, and $(B_x \cap V(G') :x\in V(T'))$ is a tree-decomposition of $G'$. By induction, there is a set $Z'$ of at most $q-1$ nodes in $T'$ such that each component of $G'-\bigcup\{B_z:z\in Z'\}$ has weight at most $\frac{\gamma(G)}{q+1}$. Let $Z := Z' \cup  \{v\}$. Each component of $G-\bigcup\{B_z:z\in Z\}$ is a component of either $G'-\bigcup\{B_x:z\in Z'\}$ or $G_v-B_v$. The former components have weight at most $\frac{\gamma(G)}{q+1}$ by induction. The latter components have weight at most $\frac{\gamma(G)}{q+1}$ by the choice of $v$. Thus each component of $G-\bigcup\{B_z:z\in Z\}$ has weight at most $\frac{\gamma(G)}{q+1}$.
\end{proof}

\cref{TreewidthSep} implies the next result, where each vertex in $S$ is weighted 1, and each vertex in $V(G)\setminus S$ is weighted 0.

\begin{cor}
\label{TreewidthSetSep}
For any tree-decomposition $\TT$ of a graph $G$, for any set $S\subseteq V(G)$, for any integer $q\geq 0$, there is a set $X\subseteq V(G)$ consisting of the union of at most $q$ bags of $\TT$, such that each component of $G-X$ has at most $\frac{|S|}{q+1}$ vertices in $S$. In particular, if $\TT$ has width $k$, then $|X|\leq q(k+1)$.
\end{cor}

We use \cref{TreewidthSetSep} in the proof of \cref{Spread} below. 

The next lemma by \citet{RS-II} builds on  \cref{RS-BasicSeparator}  by combining the components of $G-X$ into two groups. 

\begin{lem}[{\protect\citep[(2.6)]{RS-II}}]
\label{RS-CombinedSeparator}
For any graph $G$ with treewidth at most $k$, for any set $S\subseteq V(G)$, there are induced subgraphs $G_1$ and $G_2$ of $G$ with $G_1\cup G_2=G$, such that if $X:=V(G_1\cap G_2)$, then $|X|\leq k+1$ and $G_i-X$ has at most $\frac23 |S\setminus X|$ vertices in $S$, for each $i\in\{1,2\}$.
\end{lem}

Consider the following more general `component grouping' lemma.

\begin{lem}
\label{PseudoComponents}
For any graph $G$, for any weighting $\gamma$ of $G$, for any real number $w>0$, if there is a set $X\subseteq V(G)$ such that each component of $G-X$ has weight at most $w$, then there are subgraphs $G_1,\dots,G_m$ of $G$ such that: 
\begin{itemize}
  \item $G=G_1\cup\dots\cup G_m$,
  \item $V(G_i\cap G_j)=X$ for all distinct $i,j\in\{1,\dots,m\}$,
  \item $\gamma(G_i-X) \leq w$ for each $i\in\{1,\dots,m\}$, and
  \item $m\leq \ceil{\frac{2\,\gamma(G-X)}{w}}-1$.
\end{itemize}
\end{lem}

\begin{proof} 
Say a \defn{pseudo-component} of $G-X$ is a non-empty union of components of $G-X$. Let $C_1,\dots,C_m$ be pseudo-components of $G-X$, such that $V(C_1),\dots,V(C_m)$ is a partition of $V(G-X)$, each $C_i$ has weight at most $w$, and with $m$ minimum. This is well-defined, since the components of $G-X$ are candidates. Let $G_i:=G[V(C_i)\cup X]$ for each $i\in\{1,\dots,m\}$. The first three claims hold by construction. It remains to bound $m$. By the minimality of $m$, for any distinct $i,j\in\{1,\dots,m\}$,  $\gamma(C_i)+\gamma(C_j)>w$, otherwise $C_i$ and $C_j$ could be replaced by $C_i\cup C_j$ in the list of pseudo-components. Thus $$(m-1)\gamma(G-X) = (m-1)\sum_i\gamma(C_i)= \sum_{i\neq j}\gamma(C_i)+\gamma(C_j)> \tbinom{m}{2}\,w.$$ Hence $m<\frac{2\,\gamma(G-X)}{w}$ and $m\leq \ceil{\frac{2\,\gamma(G-X)}{w}}-1$. 
\end{proof}

\begin{lem}
\label{GenSeparation} 
For any tree-decomposition $\TT$ of a graph $G$, for any weighting $\gamma$ of $G$, for any real number $\beta>0$, 
there is a set $X\subseteq V(G)$ consisting of the union of at most $\ceil{\frac{1}{\beta}}-1$ bags of $\TT$, and 
there are subgraphs $G_1,\dots,G_m$ of $G$ with $m\leq \ceil{\frac{2}{\beta}}-1$ such that: 
\begin{itemize}
  \item $G=G_1\cup\dots\cup G_m$,
  \item $V(G_i\cap G_j)=X$ for all distinct $i,j\in\{1,\dots,m\}$,
  \item $\gamma(G_i-X) \leq \beta\,\gamma(G)$ for each $i\in\{1,\dots,m\}$.
\end{itemize}
In particular, if $\TT$ has width $k$, then $|X|\leq (\ceil{\frac{1}{\beta}}-1)(k+1)$.
\end{lem}

\begin{proof}
Let $w:=\beta\,\gamma(G)$ and $q:=\ceil{\frac{1}{\beta}}-1$. 
So $q\geq 0$ and $\beta\geq\frac{1}{q+1}$. 
By \cref{TreewidthSep}, there is a set $X$ of at most $q(k+1)$ vertices in $G$ such that each component of $G-X$ has weight at most $\frac{\gamma(G)}{q+1}\leq w$. 
The result follows from \cref{PseudoComponents}, where 
$m
\leq  \ceil{\frac{2\,\gamma(G-X)}{\beta\gamma(G)}}-1
\leq  \ceil{\frac{2}{\beta}}-1$. 
\end{proof}

\cref{GenSeparation} implies the next result, where each vertex in $S$ is weighted 1, and each vertex in $V(G)\setminus S$ is weighted 0.

\begin{cor}
\label{SeparationSetSep}
For any tree-decomposition $\TT$ of a graph $G$, for any set $S\subseteq V(G)$, for any real number $\beta>0$, 
there is a set $X\subseteq V(G)$ consisting of the union of at most $\ceil{\frac{1}{\beta}}-1$ bags of $\TT$, 
and there are subgraphs $G_1,\dots,G_m$ of $G$ with $m\leq \ceil{\frac{2}{\beta}}-1$ such that: 
\begin{itemize}
  \item $G=G_1\cup\dots\cup G_m$,
  \item $V(G_i\cap G_j)=X$ for all distinct $i,j\in\{1,\dots,m\}$,
  \item $G_i-X$ has at most $\beta|S|$ vertices in $S$ for each $i\in\{1,\dots,m\}$.
\end{itemize}
In particular, if $\TT$ has width $k$, then $|X|\leq (\ceil{\frac{1}{\beta}}-1)(k+1)$.
\end{cor}

The case $\beta=\frac23$ and $m=2$ of \cref{SeparationSetSep} almost implies \cref{RS-CombinedSeparator}; the only difference is that in \cref{RS-CombinedSeparator}, each $G_i-X$ has at most $\frac{2}{3}|S\setminus X|$ vertices in $S$. 

We finish this section by noting that balanced separators like in \cref{RS-BasicSeparator} characterise treewidth up to a constant factor, as shown by the following result (see \citep{Reed97,RS-X,CDDFGHHWWY}).

\begin{thm}
\label{Bruce}
Let $k$ be a positive integer. Let $G$ be a graph such that for every set $S$ of $2k+1$ vertices in $G$ there is a set $X$ of $k$ vertices in $G$ such that each component of $G-X$ has at most $k$ vertices in $S$. Then $G$ has treewidth at most $3k$.
\end{thm}

Also note the following qualitative strengthening of \cref{Bruce} by \citet{DN19} (not used in this paper). 

\begin{thm}[\citep{DN19}] 
Let $G$ be a graph such that for every subgraph $G'$ of $G$ there is a set $X$ of at most $k$ vertices in $G'$ such that each component of $G'-X$ has at most $\frac12|V(G')|$ vertices. Then $G$ has treewidth at most $15k$.
\end{thm}

\section{Small Spread and Degree}
\label{SmallSpread}

This section proves \cref{SpreadDegree}, which shows that every graph has a tree-decomposition with small width, small spread and small degree. A key idea is the following sufficient condition  for small spread. A tree-decomposition $(B_x:x\in V(T))$ is \defn{rooted} if $T$ is rooted. A rooted tree-decomposition $(B_x:x\in V(T))$ is \defn{slick} if for each edge $xy\in E(T)$ with $x$ the parent of $y$, for each vertex $v\in B_x\cap B_y$, we have $ (N_G(v) \cap B_y) \setminus B_x \neq\emptyset$.

\begin{lem}
\label{SlickSpread}
In a slick tree-decomposition $(B_x:x\in V(T))$ of a graph $G$, each vertex $v\in V(G)$ has spread at most $\deg_G(v)+1$.
\end{lem}

\begin{proof}
Consider a vertex $v\in V(G)$. Let $T_v:=T[\{x\in V(T):v\in B_x\}]$. For each edge $xy\in E(T_v)$ with $x$ the parent of $y$, there is a vertex $\hat{y}\in (N_G(v)\cap B_y)\setminus B_x$. Consider distinct non-root nodes $y_1,y_2\in V(T_v)$. Without loss of generality, the parent $x_1$ of $y_1$ is on the $y_1y_2$-path in $T$. Since  $\hat{y_1}\not\in B_{x_1}$ and $T_{\hat{y_1}}$ is connected, $\hat{y_1}\neq \hat{y_2}$. Thus $T_v$ has at most $\deg_G(v)$ non-root nodes, and $|V(T_v)|\leq\deg_G(v)+1$, as desired. 
\end{proof}

The next lemma (which essentially adds the `slick' property to \cref{Bruce}) is the main tool for proving \cref{SpreadDegree}.

\newcommand{\myalpha}{t}

\begin{lem}
\label{SpreadLemma}
Let $\ell,\myalpha$ be positive integers. Let $G$ be a graph such that for every set $S$ of $2\myalpha+2\ell$ vertices in $G$ there is a set $X$ of at most $\ell$ vertices in $G$, such that each component of $G-X$ has at most $\myalpha$ vertices in $S$. Then $G$ has a slick tree-decomposition with width at most $2\myalpha+3\ell-1$ and degree at most $4+\ceil{\frac{4\ell}{\myalpha}}$.
\end{lem}

\cref{SpreadLemma} is implied by the following slightly stronger statement.

\begin{lem}
\label{AlphaBeta}
Let $\ell,\myalpha$ be positive integers. Let $G$ be a graph such that for every set $S$ of $2t+2\ell$ vertices in $G$, there is a set $X$ of at most $\ell$ vertices in $G$, such that each component of $G-X$ has at most $\myalpha$ vertices in $S$. Then for every set $R$ of at most $2\myalpha+2\ell$ vertices in $G$ there is a slick tree-decomposition $(B_x:x\in V(T))$ of $G$ rooted at $r\in V(T)$ such that $R\subseteq B_r$, and $|B_x|\leq 2\myalpha+3\ell$ for each $x\in V(T)$. Moreover,  $\Delta(T)\leq 4+\ceil{\frac{4\ell}{\myalpha}}$ and $\deg_T(r)\leq 3+\ceil{\frac{4\ell}{\myalpha}}$.
\end{lem}

\begin{proof}
We proceed by induction on $|V(G)|$. In the base case, if $|V(G)|\leq 2\myalpha+3\ell$, then the tree-decomposition with one bag $V(G)$ satisfies the claim. Now assume that $|V(G)| > 2\myalpha+3\ell$. Adding vertices if necessary, we may assume that $|R|=2\myalpha+2\ell$. By assumption, there is a set $X$ of at most $\ell$ vertices in $G$, such that each component of $G-X$ has at most $\myalpha$ vertices in $R$. 

Weight each vertex in $R$ by 1, and weight each vertex in $V(G)\setminus R$ by  0. 
The total weight is $|R|$, and each component of $G-X$ has weight at most $\myalpha$. 
By \cref{PseudoComponents} with $w=\myalpha$, there are subgraphs $G_1,\dots,G_m$ of $G$ such that: 
\begin{itemize}
  \item $G=G_1\cup\dots\cup G_m$,
  \item $V(G_i\cap G_j)=X$ for all distinct $i,j\in\{1,\dots,m\}$,
  \item $G_i-X$ has at most $\myalpha$ vertices in $R$, for each $i\in\{1,\dots,m\}$, and 
  \item $m\leq \ceil{\frac{2\,|R\setminus X|}{\myalpha}}-1 
  \leq \ceil{\frac{4t+4\ell}{\myalpha}}-1 = 3+\ceil{\frac{4\ell}{\myalpha}}$.
\end{itemize}
Note that $2\myalpha+2\ell=|R|\leq |X|+\myalpha m \leq \ell+\myalpha m$, implying $(m-2)\myalpha \geq \ell\geq 1$ and $m\geq 3$.

Consider $i\in\{1,\dots,m\}$. 
Let $R_i:=X\cup (R \cap V(G_i))$. 
Note that $|R_i|\leq \myalpha+\ell$.
Let $R^-_i$ be the set of vertices $v\in R_i$ such that $N_{G_i}(v)\subseteq R_i$. 
For each vertex $v\in R_i\setminus R^-_i$, we have 
$N_{G_i}(v)\setminus R_i\neq\emptyset$.
Let $R'_i$ be obtained from $R_i\setminus R^-_i$ by adding one vertex in 
$N_{G_i}(v)\setminus R_i$ to $R'_i$, for each $v\in R_i\setminus R^-_i$. 
So $|R'_i|\leq 2|R_i\setminus R^-_i|\leq 2|R_i|\leq 2(\myalpha+\ell)$. 
Let $G'_i:= G_i- R^-_i$, so $R'_i \subseteq V(G'_i)$. 
Since $m\geq 3$, we have $|V(G'_i)|<|V(G)|$. 

We now show the separator assumption is passed from $G$ to $G'_i$. Let $S$ be a set of $2\myalpha+2\ell$ vertices in $G'_i$. By assumption, there is a set $X$ of at most $\ell$ vertices in $G$ such that each component of $G-X$ has at most $\myalpha$ vertices in $S$. Each component of $G'_i-X$ is a subgraph of a component of $G-X$. So each component of $G'_i-X$ has at most $\myalpha$ vertices in $S$. 

By induction, there is a slick tree-decomposition $(B^i_x:x\in V(T_i))$ of $G'_i$ rooted at $r_i\in V(T)$ such that $R'_i\subseteq B_{r_i}$, and $|B^i_x|\leq 2\myalpha+3\ell$ for each $x\in V(T_i)$. Moreover, $\Delta(T_i)\leq 4+\ceil{\frac{4\ell}{\myalpha}}$ and $\deg_T(r_i)\leq 3+\ceil{\frac{4\ell}{\myalpha}}$.

Let $T$ be obtained from the disjoint union $T_1\cup\dots\cup T_m$ by adding one new node $r$ adjacent to $r_1,\dots,r_m$. 
Root $T$ at $r$. 
Let $B_r:= X\cup R$, so $|B_r|\leq 2\myalpha+3\ell$ and $R\subseteq B_r$, as desired. 
We now show that $(B_x:x\in V(T))$ is a tree-decomposition of $G$.
The vertex-property holds since any vertex in at least two of $G_1,\dots,G_m$ is also in $B_r$. 
Consider an edge $vw\in E(G)$. If $v,w\in X\cup R$ or $v,w\in V(G_i)$, then $v,w$ are in a common bag. Otherwise, $v\in X\cup R$ and $w$ is in some $G'_i$. Thus $v\in (R_i\cup X)\setminus R^-_i$ implying $v\in R'_i$. 
Hence $v$ and $w$ are in the bag $B_{r_i}$. 
So $(B_x:x\in V(T))$ is a tree-decomposition of $G$.
By construction, $\Delta(T)\leq 4+\ceil{\frac{4\ell}{\myalpha}}$ and $\deg_T(r)=m\leq 3+\ceil{\frac{4\ell}{\myalpha}}$. 

The slick property holds for every edge in $T_1\cup\dots\cup T_m$ by induction. 
Consider an edge $rr_i$ of $T$ and a vertex $v\in B_r\cap B_{r_i}$, for some $i\in\{1,\dots,m\}$. 
Thus $v\in R'_i$ and $v\not\in R^-_i$. 
Hence there is a vertex in $N_{G_i}(v)\setminus R_i$ which was added to $R'_i$, and is therefore in $B_{r_i}$. 
Hence $(B_x:x\in V(T))$ is slick.   
\end{proof}

The next theorem and \cref{SlickSpread} imply \cref{SpreadDegree} (which implies \cref{Spread}).

\begin{thm}
\label{SlickMain}
Every graph $G$ with treewidth at most $k$ has a slick tree-decomposition with width at most $14k+13$ and degree at most 6.
\end{thm}

\begin{proof}
By \cref{TreewidthSetSep} with $q=2$, for every set $S\subseteq V(G)$ there is a set $X$ of at most $2(k+1)$ vertices such that each component of $G-X$ has at most $\frac{1}{3}|S|$ vertices in $S$. Let $\ell:=2(k+1)$. In particular, if $|S|=6\ell$ then there is a set $X$ of at most $\ell$ vertices in $G$, such that each component of $G-X$ has at most $2\ell$ vertices in $S$.
Hence \cref{SpreadLemma} is applicable with $\myalpha=2\ell$.  Therefore $G$ has a slick tree-decomposition with width at most $2\myalpha+3\ell-1=7\ell-1=14k+13$ and degree at most $4+\ceil{\frac{4\ell}{\myalpha}}=6$. 
\end{proof}

\section{Small Order}
\label{SmallOrder}

This section proves \cref{Small} showing that every graph has a tree-decomposition with small width and small order. 

\begin{lem}
\label{FindSubtree}
For every rooted tree $T$ and integer $k\in\{2,\dots, |V(T)|\}$, there is a rooted subtree $T'$ of $T$ such that $|V(T')|\in\{k,\dots,2k-2\}$ and the root of $T'$ is the only vertex of $T'$ possibly adjacent to vertices in $T-V(T')$.
\end{lem}

\begin{proof}
Let $r$ be the root of $T$. For each vertex $v$ of $T$, let $T_v$ be the subtree of $T$ induced by $v$ and the descendants of $v$. Let $v$ be a vertex in $T$ at maximum distance from $r$ such that $|V(T_v)|\geq k$. This is well-defined since $|V(T_r)|=|V(T)| \geq k$. Let $w_1,\dots,w_d$ be the children of $v$. So $d\geq 1$, since $|V(T_v)|\geq k\geq 2$. By the choice of $v$, $|V(T_{w_i})|\leq k-1$ for each $i\in\{1,\dots,d\}$, and $\sum_{i=1}^d|V(T_{w_i})|\geq k-1$. 
There exists a minimum integer $c\in\{1,\dots,d\}$ such that 
$\sum_{i=1}^c|V(T_{w_i})|\geq k-1$. So 
$\sum_{i=1}^{c-1}|V(T_{w_i})|\leq k-2$ and
$\sum_{i=1}^{c}|V(T_{w_i})|\leq 2k-3$.
Let $T':= T[ \bigcup_{i=1}^{c} V(T_{w_i})\cup\{v\}]$.
So $|V(T')|\in\{k,\dots,2k-2\}$. 
By construction, $v$ is the root of $T'$, and $v$ is the only vertex in $T'$ possibly adjacent to vertices in $T-V(T')$.
\end{proof}

Let $T$ be a tree rooted at a vertex $r\in V(T)$. As illustrated in \cref{TreeDivision}, a \defn{division} of $T$ is a sequence $(T_1,\dots,T_m)$ of pairwise edge-disjoint subtrees of $T$ such that: \begin{itemize}
\item $T=T_1\cup\dots\cup T_m$, 
\item $r\in V(T_1)$, 
\item for $i\in\{2,\dots,m\}$, 
if $r_i$ is the root of $T_i$ then $V(T_i)\cap V(T_1\cup\dots\cup T_{i-1})=\{r_i\}$.
\end{itemize}

\begin{figure}[ht]
\psset{unit=10mm}
\begin{pspicture}(-2,3.9)(13,9.2)

\rput(4.5,9){$r$}
		
\pspolygon[linecolor=blue,fillcolor=yellow,fillstyle=solid](5,9)(3,7)(7,7)
\pscircle[linewidth=0.5pt,linecolor=darkgray,fillcolor=lime,fillstyle=solid](5,9){0.1}
\rput(5,7.5){$T_1$}

\pspolygon[linecolor=blue,fillcolor=yellow,fillstyle=solid](3,7)(2,5)(4,5)
\pscircle[linewidth=0.5pt,linecolor=darkgray,fillcolor=lime,fillstyle=solid](3,7){0.1}
\rput(3,5.5){$T_2$}

\pspolygon[linecolor=blue,fillcolor=yellow,fillstyle=solid](6,8)(9.5,6)(12.5,6)
\pscircle[linewidth=0.5pt,linecolor=darkgray,fillcolor=lime,fillstyle=solid](6,8){0.1}
\rput(9.35,6.5){$T_3$}

\pspolygon[linecolor=blue,fillcolor=yellow,fillstyle=solid](6,7)(5,4)(6.5,4)
\pscircle[linewidth=0.5pt,linecolor=darkgray,fillcolor=lime,fillstyle=solid](6,7){0.1}
\rput(5.75,4.75){$T_4$}

\pspolygon[linecolor=blue,fillcolor=yellow,fillstyle=solid](4.5,7)(4,5.5)(5,5.5)
\pscircle[linewidth=0.5pt,linecolor=darkgray,fillcolor=lime,fillstyle=solid](4.5,7){0.1}
\rput(4.5,6){$T_{12}$}

\pspolygon[linecolor=blue,fillcolor=yellow,fillstyle=solid](2.45,6)(-1,4)(1,4)
\pscircle[linewidth=0.5pt,linecolor=darkgray,fillcolor=lime,fillstyle=solid](2.45,6){0.1}
\rput(0.8,4.5){$T_5$}

\pspolygon[linecolor=blue,fillcolor=yellow,fillstyle=solid](6,7)(7,5)(9,5)
\pscircle[linewidth=0.5pt,linecolor=darkgray,fillcolor=lime,fillstyle=solid](6,7){0.1}
\rput(7.5,5.5){$T_6$}

\pspolygon[linecolor=blue,fillcolor=yellow,fillstyle=solid](3,5)(2,4)(4,4)
\pscircle[linewidth=0.5pt,linecolor=darkgray,fillcolor=lime,fillstyle=solid](3,5){0.1}
\rput(3,4.35){$T_7$}

\pspolygon[linecolor=blue,fillcolor=yellow,fillstyle=solid](8,5)(7.25,4)(8.75,4)
\pscircle[linewidth=0.5pt,linecolor=darkgray,fillcolor=lime,fillstyle=solid](8,5){0.1}
\rput(8,4.35){$T_8$}

\pspolygon[linecolor=blue,fillcolor=yellow,fillstyle=solid](10.75,6)(9,4)(10.5,4)
\pscircle[linewidth=0.5pt,linecolor=darkgray,fillcolor=lime,fillstyle=solid](10.75,6){0.1}
\rput(10,4.35){$T_9$}

\pspolygon[linecolor=blue,fillcolor=yellow,fillstyle=solid](4,8)(-1,6)(1.2,6)
\pscircle[linewidth=0.5pt,linecolor=darkgray,fillcolor=lime,fillstyle=solid](4,8){0.1}
\rput(1,6.25){$T_{10}$}

\pspolygon[linecolor=blue,fillcolor=yellow,fillstyle=solid](10.75,6)(11,4)(12.5,4)
\pscircle[linewidth=0.5pt,linecolor=darkgray,fillcolor=lime,fillstyle=solid](10.75,6){0.1}
\rput(11.4,4.35){$T_{11}$}

\end{pspicture}
   \caption{Example of a tree division.}
    \label{TreeDivision}
\end{figure}
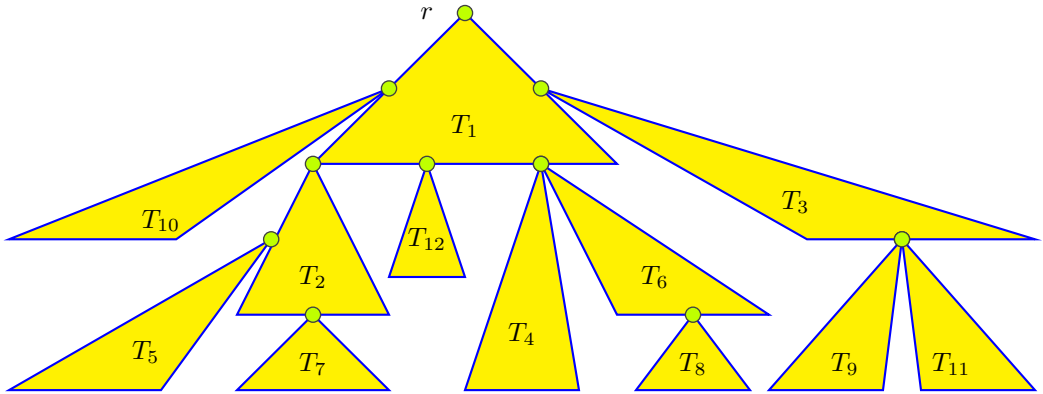

\begin{lem}
\label{PartitionTree}
For any integer $k\geq 2$, every rooted tree $T$ with $|V(T)|\geq k$ has a division $(T_1,\dots,T_m)$ such that $m\leq \frac{|V(T)|}{k-1}$, and
$|V(T_i)|\in\{k,\dots,2k-2\}$  for each $i\in\{1,\dots,m\}$. 
\end{lem}

\begin{proof}
We proceed by induction on $|V(T)|$ with $k$ fixed. 
If $|V(T)|=k$ then the claim holds with $T_1:=T$ and $m:=1$.
Now assume that $|V(T)|\geq k+1$. 
By \cref{FindSubtree}, 
there is a rooted subtree $T'$ of $T$ such that $|V(T')|\in\{k,\dots,2k-2\}$ and the root $v$ of $T'$ is the only vertex in $T'$ possibly adjacent to vertices in $T-V(T')$.
Let $T'':= T-(V(T')\setminus\{v\})$, which is a subtree of $T$ with at most $|V(T)|-(k-1)$ vertices, and $r\in V(T'')$. 
By induction, $T''$ has a division  $(T_1,\dots,T_{m'})$ such that $m' \leq \frac{|V(T)|-(k-1)}{k-1}$, and $|V(T_i)|\in\{k,\dots,2k-2\}$ for each $i\in\{1,\dots,m'\}$. 
Let $m:=m'+1\leq \frac{|V(T)|}{k-1}$.
Let $T_m:=T'$.
So $(T_1,\dots, T_m)$ is a division of $T$, 
and $|V(T_i)|\in\{k,\dots,2k-2\}$  for each $i\in\{1,\dots,m\}$.
\end{proof}

Let $(B_x:x\in V(T))$ be a tree-decomposition of a graph $G$, where $T$ is a tree rooted at $r\in V(T)$. Let $(T_1,\dots,T_m)$ be a division of $T$, where $T_i$ is rooted at $r_i$. Let $F$ be a tree with vertex-set $\{1,\dots,m\}$, rooted at vertex $1$, where for $i\in\{2,\dots,m\}$, the parent of $i$ is any number $\alpha\in\{1,\dots,i-1\}$ such that $r_i\in V(T_\alpha)$. This is well-defined by the third property of division. 
Let $C_1:=\bigcup\{B_x:x\in V(T_1)\}$, and for $i\in \{2,\dots,m\}$, 
let $C_i:= \bigcup\{B_x:x\in V(T_i)\setminus\{r_i\}\}$. 
Then $(C_i:i\in V(F))$ is called the \defn{quotient} of 
$(B_x:x\in V(T))$ with respect to $(T_1,\dots,T_m)$. 

\begin{lem}
\label{QuotientTreeDecomp}
Under the above definitions, the quotient $(C_i:i\in V(F))$ is a tree-decomposition of $G$. 
\end{lem}

\begin{proof}
For each node $x\in V(T)$, let $i(x):=\min\{i\in\{1,\dots,m\}:x\in V(T_i)\}$. Note that each node $x\in V(T)\setminus\{r\}$ is not the root of $T_{i(x)}$, since $r$ is the root of $T_1$, and if $i(x)\geq 2$ then $r_{i(x)}$ is in some tree $T_j$ with $j<i(x)$. 

We now prove that  $(C_i:i\in V(F))$ has the edge-property of tree-decompositions. For each edge $vw\in E(G)$ there is a node $x\in V(T)$ with $v,w\in B_x$. If $x=r$ then $v,w\in C_1$, as desired. If $x\neq r$, then $x$ is not the root of $T_{i(x)}$, implying $v,w\in C_{i(x)}$, as desired. 

We now prove that  $(C_i:i\in V(F))$ has the vertex-property of tree-decompositions. Consider a vertex $v\in V(G)$. 
Let $Y_v$ be the subgraph of $F$ induced by $\{i\in V(F):v\in C_i\}$.
We first show that $Y_v$ is non-empty. 
There is a node $x\in V(T)$ with $v\in B_x$. 
If $x=r$ then $v\in C_1$, as desired. 
If $x\neq r$, then $x$ is not the root of $T_{i(x)}$, implying $v\in C_{i(x)}$, as desired. So $Y_v$ is non-empty. 
We now show that  $Y_v$ is connected. 
Suppose that $Y_v$ is disconnected. 
Let $i$ and $j$ be the root vertices of distinct components of $Y_v$. 
Without loss of generality, $1\leq j<i$. 
Since $i$ is in $Y_v$ and $i\geq 2$, 
there is a node $x$ in $T_{i}-r_i$ with $v\in B_{x}$. 
Similarly, since $j$ is in $Y_v$, 
there is a node $y$ in $T_{j}$ with $v\in B_{y}$. 
Since $Y_v$ is an induced subgraph of $F$, and $j<i$, 
the parent $\alpha$ of $i$ is on the $ij$-path in $F$. 
Since $Y_v$ is an induced subgraph of $F$, and $i$ is the root of its component, $\alpha$ is not in $Y_v$. 
By construction, $r_{i}\in V(T_\alpha)$. 
So $v\not\in B_{r_{i}}$.
Since $\alpha$ is on the $ij$-path in $F$, 
$r_{i}$ is on the $xy$-path in $T$, which contradicts the vertex-property for the tree-decomposition $(B_x:x\in V(T))$ for vertex $v$. Thus $Y_v$ is connected. 

 So $(C_i:i\in V(F))$ is a tree-decomposition of $G$. 
\end{proof}


\begin{thm}
\label{SmallTreeDecomp}
For any graph $G$ and integer $k\geq\max\{\tw(G),1\}$, there is a tree-decomposition of $G$ with width at most $3k-1$ and order at most 
$\max\{\frac{|V(G)|}{k}-1,1\}$.
\end{thm}

\begin{proof}
If $|V(G)| \leq 2k$ then the tree-decomposition with one bag $V(G)$ satisfies the claim. Now assume that $|V(G)|>2k$. It is well-known that $G$ has a tree-decomposition $(B_x:x\in V(T))$ with width $k$ such that $|V(T)|=|V(G)|-k$, and $|B_x\setminus B_y|=|B_y\setminus B_x|=1$ for each edge $xy\in E(T)$ (see \citep{DvoWoo} for a proof). Root $T$ at an arbitrary node $r\in V(T)$. For each non-root node $x\in V(T)$ with parent $y\in V(T)$, there is exactly one vertex $v_x$ in $B_x\setminus B_y$. 
By \cref{PartitionTree} (applied with $k+1$), 
$T$ has a division $(T_1,\dots,T_m)$ such that $m\leq \frac{|V(T)|}{k} = \frac{|V(G)|-k}{k}$, and $|V(T_i)|\in\{k+1,\dots,2k\}$ for each $i\in\{1,\dots,m\}$. By \cref{QuotientTreeDecomp}, the quotient 
 $(C_i:i\in V(F))$ of $(B_x:x\in V(T))$ with respect to $(T_1,\dots,T_m)$ is a tree-decomposition of $G$. 
 For each $i\in V(F)$, $C_i$ is contained in the union of $B_{r_i}$ and the set of vertices $v_x$ where $x$ is a non-root vertex in $T_i$. So $|C_i|\leq (k+1)+|V(T_i)|-1\leq 3k$. Hence, $(C_i:i\in V(F))$ is a tree-decomposition of $G$ with width at most $3k-1$, where $|V(F)|=m\leq \frac{|V(G)|}{k}-1$.
\end{proof}

\section{Small Spread and Order}
\label{SmallSpreadOrder}

This section combines the previous proof methods to establish  \cref{SpreadAndSmall}, which shows that every graph has a tree-decomposition with small width, small spread, and few bags. We start with a weighted version of \cref{FindSubtree}. 

%

\begin{lem}
\label{FindWeightedSubtreePlus}
Let $T$ be a rooted tree with weighting $\gamma:V(T)\rightarrow\{1,2,\dots,k-1\}$ for some integer $k\geq 2$ with $\gamma(T)\geq 2k-2$.
Then there is a subtree $T'$ of $T$ rooted at some vertex $v$ such that:
\begin{itemize}
\item $\gamma(T')\in\{k,\dots,4k-6\}$, 
\item $v$ is the only vertex of $T'$ possibly adjacent to vertices in $T-V(T')$, 
\item $\gamma(T'-v)\in\{k-1,\dots,3k-5\}$.
\end{itemize}
\end{lem}

\begin{proof}
Let $r$ be the root of $T$. For each vertex $v$ of $T$, let $T_v$ be the subtree of $T$ induced by $v$ and the descendants of $v$. Let $v$ be a vertex in $T$ at maximum distance from $r$ such that $\gamma(T_v)\geq k-1+\gamma(v)$. This is well-defined since $$\gamma(T_r)=\gamma(T)\geq 2k-2 \geq k-1 + \gamma(r).$$
Since $\gamma(T_v)\geq k-1+\gamma(v)\geq k$ and $\gamma(v)\leq k-1$, $v$ is not a leaf of $T$. Let $w_1,\dots,w_d$ be the children of $v$, where $d\geq 1$. 
By the choice of $v$,  for each $i\in\{1,\dots,d\}$, 
$$\gamma(T_{w_i}) \leq k-2+\gamma(w_i) \leq 2k-3,$$
and $$\sum_{i=1}^d \gamma(T_{w_i}) = \gamma(T_v)-\gamma(v) \geq k-1.$$ 
There exists a minimum integer $c\in\{1,\dots,d\}$ such that 
$\sum_{i=1}^c\gamma(T_{w_i})\geq k-1$. 
Let $T':= T[ \bigcup_{i=1}^{c} V(T_{w_i})\cup\{v\}]$.
Note that 
$\gamma(T'-v) = \sum_{i=1}^c \gamma(T_{w_i}) \geq k-1$. 
For an upper bound, by the choice of $c$, 
$$\gamma(T'-v)=\gamma(T_{w_c}) + \sum_{i=1}^{c-1}\gamma(T_{w_i}) \leq
(2k-3) + (k-2) \leq 3k-5.$$
Together these bounds show that
$$\gamma(T')=\gamma(T'-v)+\gamma(T)\in 
\{k-1+\gamma(v),\dots,3k-5+\gamma(v)\} 
\subseteq \{k,\dots,4k-6\}.$$
By construction, $v$ is the root of $T'$, and $v$ is the only vertex in $T'$ possibly adjacent to vertices in $T-V(T')$. 
\end{proof}

The next lemma is a weighted analogue of \cref{PartitionTree}.

\begin{lem}
\label{PartitionWeightedTree}
Let $T$ be a rooted tree with weighting $\gamma:V(T)\rightarrow \{1,\dots,k-1\}$ for some integer $k\geq 2$ with $\gamma(T)\geq 2k-2$. 
Then $T$ has a division $(T_1,\dots,T_m)$ such that:
\begin{itemize}
    \item $m\leq \frac{\gamma(T)}{k-1}$, 
    \item for each $i\in\{1,\dots,m\}$, $\gamma(T_i)\in\{k,\dots,5k+2\}$, 
    \item for each $i\in\{2,\dots,m\}$, 
    if $r_i$ is the root of $T_i$, then $\gamma(T_i-r_i)\in\{k-1,\dots,3k-5\}$.
\end{itemize}
\end{lem}

\begin{proof}
We proceed by induction on $\gamma(T)$ with $k$ fixed. 
If $\gamma(T)\leq 5k+2$ then the claim holds with $T_1:=T$ and $m:=1$.
Now assume that $\gamma(T)\geq 5k+3$. 
By \cref{FindWeightedSubtreePlus}, 
there is a subtree $T'$ of $T$ rooted at some vertex $v$ such that:
\begin{itemize}
\item $\gamma(T')\in\{k,\dots,4k-6\}$, 
\item $v$ is the only vertex of $T'$ possibly adjacent to vertices in $T-V(T')$, 
\item $\gamma(T'-v)\in\{k-1,\dots,3k-5\}$.
\end{itemize}
Let $T'':= T-(V(T')\setminus\{v\})$, which is a subtree of $T$ with $r\in V(T'')$. Note that 
\begin{align*}
\gamma(T'') & = \gamma(T)-\gamma(T'-v)\leq\gamma(T)-(k-1) \text{ and}\\
\gamma(T'') & =\gamma(T)-\gamma(T'-v) \geq (5k+3)-(3k-5) \geq 2k-2.
\end{align*}
By induction, $T''$ has a division $(T_1,\dots,T_{m'})$ such that:
\begin{itemize}
    \item $m' \leq \frac{\gamma(T'')}{k-1} \leq \frac{\gamma(T)-(k-1)}{k-1}$, 
    \item for each $i\in\{1,\dots,m'\}$, $\gamma(T_i)\in\{k,\dots,5k+2\}$, 
 \item for each $i\in\{2,\dots,m'\}$, 
    if $r_i$ is the root of $T_i$, then $\gamma(T_i-r_i)\in\{k-1,\dots,3k-5\}$.
\end{itemize}
Let $m:=m'+1\leq \frac{\gamma(T)}{k-1}$.
Let $T_m:=T'$.
So $(T_1,\dots, T_m)$ is a division of $T$. The claimed properties hold since $v$ is the root of $T'$, and thus $\gamma(T_m-r_m)=\gamma(T'-v)\in \{k-1,\dots,3k-5\}$.
\end{proof}

\begin{lem}
\label{MakeSmall}
For any integer $\ell\geq 2$, if a graph $G$ with at least $2\ell-2$ vertices has a slick tree-decomposition $(B_x:x\in V(T))$ with width at most $\ell-2$, then $G$ has a slick tree-decomposition $(C_x:x\in V(F))$ with width at most $4\ell-7$ and order at most $\frac{|V(G)|}{\ell-1}$.
\end{lem}

\begin{proof}
Let $r$ be the root of $T$. 
Weight $T$ as follows. 
Let $\gamma(r):=|B_r|$. 
For each edge $xy$ in $T$ with $x$ the parent of $y$, 
let $\gamma(y):=|B_y\setminus B_x|$.
If $\gamma(y)=0$ then $B_y\subseteq B_x$, contradicting the slick property for any $v\in B_y$ (since we may assume that $B_y\neq\emptyset$).
So $\gamma(y)\geq 1$ and $\gamma(y)\leq|B_y|\leq \ell-1$. 
Note that $\gamma(T)=|V(G)|\geq 2\ell-2$.

By \cref{PartitionWeightedTree}, $T$ has a division $(T_1,\dots,T_m)$ such that:
\begin{itemize}
    \item $m\leq \frac{\gamma(T)}{\ell-1} = \frac{|V(G)|}{\ell-1}$, 
    and 
    \item for each $i\in\{2,\dots,m\}$, 
    if $r_i$ is the root of $T_i$, then $\gamma(T_i-r_i)\in\{\ell-1,\dots,3\ell-5\}$.
\end{itemize}

By \cref{QuotientTreeDecomp}, the quotient 
 $(C_i:i\in V(F))$ of $(B_x:x\in V(T))$ with respect to $(T_1,\dots,T_m)$ is a tree-decomposition of $G$. So $|V(F)|=m\leq \frac{|V(G)|}{\ell-1}$, as desired. For each $i\in V(F)$, $C_i$ is contained in the union of $B_{r_i}$ and the union of $B_y\setminus B_x$ taken over the edges $xy\in E(T_i)$ with $x$ the parent of $y$. So $|C_i|\leq (\ell-1)+\gamma(T_i-r_i)\leq 4\ell-6$, and  $(C_i:i\in V(F))$ has width at most $4\ell-7$. 

It remains to show that $(C_i:i\in V(F))$  is slick.
Consider an edge $\alpha i\in E(F)$ where $\alpha$ is the parent of $i$. Consider $v\in C_i\cap C_\alpha$. By construction, $v\in B_{r_i}$ and $v$ is in some other bag $B_y$ with $y$ a non-root node of $T_i$. Thus $v$ is in $B_y$ for some child $y$ of $r_i$. Since $(B_x:x\in V(T))$ is slick, $v$ has a neighbour $w$ in $B_y\setminus B_{r_i}$. So $w\in C_i\setminus C_\alpha$. Hence  $(C_i:i\in V(F))$  is slick.
\end{proof}

The next theorem and \cref{SlickSpread} imply \cref{SpreadAndSmall}.

\begin{thm}
\label{SlickAndSmall}
For any graph $G$ and integer $k\geq\tw(G)$, $G$ has a slick tree-decomposition with width at most $56k+53$ and order at most $\max\{\frac{|V(G)|}{14k+14},1\}$.
\end{thm}

\begin{proof}
Let $\ell:=14k+15$. By \cref{SlickMain}, $G$ has a slick tree-decomposition with width at most $14k+13=\ell-2$. If  $|V(G)|\leq 2\ell-3$ then the tree-decomposition with one bag $V(G)$ satisfies the claim. Now assume that $|V(G)|\geq 2\ell-2$. By \cref{MakeSmall}, $G$ has a slick tree-decomposition with width at most $4\ell-7=56k+53$ and order at most $\frac{|V(G)|}{\ell-1}=\frac{|V(G)|}{14k+14}$.
\end{proof}

\section{Open Problems}

We conclude with some open problems:

\begin{enumerate}[(Q1)]

\item What is the infimum of the $c\in\mathbb{R}$ such that for some $c'\in\mathbb{R}$, every graph $G$ with treewidth $k$  has a tree-decomposition with width at most $(c+o(1))k$, in which each vertex $v\in V(G)$ has spread at most $c'(\deg(v)+1)$? \cref{Spread} says the answer is at most 14.

\item We expect that $n\times n$ grid graphs imply that the answer to (Q1) is at least 2. In particular, I conjecture there no constants $\varepsilon,c>0$ such that every $n\times n$ grid graph has a tree-decomposition with width at most $(2-\varepsilon)n$ and spread at most $c$. I also conjecture that every optimal tree-decomposition of the $n\times n$ grid has very large spread. In particular, in every tree-decomposition of the $n\times n$ grid with width $n$, some vertex has spread $\Omega(n)$. 

\item What is the infimum of the $c\in\mathbb{R}$ such that for some $c'\in\mathbb{R}$, every graph $G$ has a tree-decomposition with width at most  $c'(\tw(G)+1)$ and average spread at most $c$? \cref{Small} says the answer is at most 3. 

\end{enumerate}

{\fontsize{10pt}{11pt}\selectfont
\def\soft#1{\leavevmode\setbox0=\hbox{h}\dimen7=\ht0\advance \dimen7 by-1ex\relax\if t#1\relax\rlap{\raise.6\dimen7 \hbox{\kern.3ex\char'47}}#1\relax\else\if T#1\relax \rlap{\raise.5\dimen7\hbox{\kern1.3ex\char'47}}#1\relax \else\if d#1\relax\rlap{\raise.5\dimen7\hbox{\kern.9ex \char'47}}#1\relax\else\if D#1\relax\rlap{\raise.5\dimen7 \hbox{\kern1.4ex\char'47}}#1\relax\else\if l#1\relax \rlap{\raise.5\dimen7\hbox{\kern.4ex\char'47}}#1\relax \else\if L#1\relax\rlap{\raise.5\dimen7\hbox{\kern.7ex \char'47}}#1\relax\else\message{accent \string\soft \space #1 not defined!}#1\relax\fi\fi\fi\fi\fi\fi}

}

\appendix
\section{Follow-up Work}

Following the original release of this paper \citep{Wood25}:
\begin{itemize}
\item \citet{KKKW} used the bound on the order of tree-decompositions in  \cref{Small} as a key ingredient in their construction of a universal graph for the class of $n$-vertex graphs with treewidth $k$.
\item \citet{BG26} made significant progress on Q1, showing that $c\in[2,3]$. 
\item \citet{BG26} disproved the conjecture in Q2, by showing that for any $m \geq n \geq 1$ and $\epsilon > 0$, the $(n \times m)$-grid has a tree-decomposition with width at most $(1+\epsilon)n$ in which each vertex has spread $O(1/\epsilon)$. On the other hand, \citet{Norin26} proved a lower bound in the direction of the conjecture in Q2. Namely, there exists $\delta>0$ such that for all integers $n,a \geq 1$ every tree-decomposition of the $(n\times n)$-grid with width $n+a-1$ has a vertex of spread at least $\delta n/a$. 
\item \citet{BG26} answered Q3 by showing that the answer is $c=1$. 
\item \citet{DKKW} improved the  best known upper bound on the domino treewidth to $(k+1)(\Delta+1)(8\Delta-3)-1$ (slightly improving on the result by  \citet{Bodlaender-DMTCS99} mentioned in the introduction).
\item The author recently noticed the following question of \citet{Antony20}: does every graph with treewidth $k$ and maximum degree $\Delta$ have a chordal completion with treewidth at most $f(k,\Delta)$ and maximum degree at most $g(k,\Delta)$ for some functions $f,g$? \cref{Completion} says the answer is `yes' with $f(k)\in O(k)$ and $g(k,\Delta)\in O(k\Delta)$. Up to the constant factors, these functions are best possible, since obviously $f(k)\geq k$ and \citet{DKKW} recently showed that $g(k,\Delta)\in\Omega(k\Delta)$. Moreover, \citet{DKKW} showed for $k\geq 3$ that $f(k)\geq 2k-2$ regardless of $g(k,\Delta)$. So the $O(k)$ bound on $\omega(G')$ in \cref{Completion} cannot be replaced by $2k-3$. This implies that in any result that bounds the spread by a function of the maximum degree (as in \cref{Spread}), the width bound must be at least $2k-2$. This provides further motivation for considering near-optimal width in \cref{Spread}. 
\end{itemize}
\end{document}